\documentclass[amssymb,amsfonts,refcheck,12pt,verbatim,righttag]{amsart}
\usepackage{amssymb,mathrsfs}

\usepackage{amssymb}
\usepackage{graphicx}
\usepackage{graphics,color}

\usepackage{amsfonts}
\usepackage{color}
\usepackage{amssymb}
\usepackage{cancel}
\usepackage{soul}

 \numberwithin{equation}{section} % \renewcommand{\rm}{\normalshape} %
\theoremstyle{plain}

\newtheorem{thm}{Theorem}[section]
\newtheorem{lem}[thm]{Lemma}

\newtheorem{ex}[thm]{Example}
\newtheorem{de}[thm]{Definition}

\newtheorem{ques}[thm]{Question}%[section]

\def\R {{\Bbb R}}
\def\N {{\Bbb N}}
\def\Z {{\Bbb Z}}

\def\Q {{\mathcal Q}}

\def\D {{\mathcal D}}

\newcommand{\diam}{\mathrm{diam}}

\DeclareMathOperator{\hdim}{dim_H}
\DeclareMathOperator{\dist}{dist}
\newcommand{\OSC}{\mathrm{OSC}}
\newcommand{\OSCH}{\mathrm{OSC}^{\mathrm{Hom}}}

\newcommand{\Stp}[1]{\noindent\textsc{#1}.}

\usepackage[colorlinks,citecolor=blue,pagebackref,urlcolor=black,hypertexnames=false]{hyperref}

\begin{document}
\baselineskip 13.7pt
\title{A separation property for iterated function systems of similitudes}

\author{De-Jun FENG}
\address{
Department of Mathematics\\
The Chinese University of Hong Kong\\
Shatin,  Hong Kong
}
\email{djfeng@math.cuhk.edu.hk}

\author{Huo-Jun Ruan}
\address{
School of Mathematical Sciences\\
Zhejiang University\\
Hangzhou 310027, Zhejiang\\
People's Republic of China
}
\email{ruanhj@zju.edu.cn}

\author{Ying Xiong}
\address{
Department of Mathematics\\
South China  University of Technology\\
Guangzhou 510641,  Guangdong\\
People's Republic of China
}
\email{xiongyng@gmail.com}



%\thanks { The research of Feng  was partially supported by the HKRGC GRF grant}

\date{}

\begin{abstract}
Let $E$ be the attractor of an iterated function system  $\{\phi_i(x)=\rho R_ix+a_i\}_{i=1}^N$ on $\R^d$, where $0<\rho<1$, $a_i\in \R^d$ and $R_i$ are orthogonal transformations on $\R^d$. Suppose that $\{\phi_i\}_{i=1}^N$ satisfies the open set condition, but not the strong separation condition. We show that $E$ can not be generated by any  iterated function system of similitudes satisfying the strong separation condition. This gives a partial answer to a folklore question about the separation conditions on the generating  iterated function systems of self-similar sets.
  \end{abstract}

\maketitle
\section{Introduction}
\label{S1}

In this paper, we investigate the separation conditions on iterated function systems of similitudes.

By an {\it iterated function system} (IFS) on $\R^d$ we mean a finite family $\Phi=\{\phi_i\}_{i=1}^N$ of uniformly contracting mappings on $\R^d$ with $N>1$. It is well known \cite{Hut1981} that for each IFS $\Phi=\{\phi_i\}_{i=1}^N$ on $\R^d$, there is a unique non-empty compact set $E\subset \R^d$ such that $$E=\bigcup_{i=1}^N \phi_i(E).$$
We call $E$ the {\it attractor} of $\Phi$. If each map $\phi_i$ in $\Phi$ is a similitude, i.e., $\phi_i$ is of the form $$\phi_i(x)=\rho_i R_ix+a_i,$$
where $0<\rho_i<1$, $a_i\in \R^d$ and $R_i$ is an orthogonal transformation on $\R^d$, we say that $E$ is a {\it self-similar set} generated by $\Phi$. The study of IFSs and their attractors is an important subject in fractal geometry, dynamical systems and probability (see e.g.~\cite{BSS2021, BishopPeres2017,Falconer2003,Hut1981}).

One of the most fundamental conditions on  IFSs of similitudes is the {\it open set condition} (OSC), under
which the dimensions of the self-similar sets and the multifractal structure of the
self-similar measures are well understood (\cite{Falconer2003, Hut1981, Olsen1995, Patzschke1997}).
Recall that an IFS $\Phi=\{\phi_i\}_{i=1}^N$ of similitudes on $\R^d$ is said to satisfy the  open set condition if there is a non-empty open set $U\subset \R^d$ such that $\phi_i(U)$, $i=1,\ldots, N$, are disjoint subsets of $U$. Another commonly used  separation condition on IFSs is the so-called {\it strong separation condition} (SSC). Recall that $\Phi=\{\phi_i\}_{i=1}^N$ is said to satisfy  the  strong separation condition if $\phi_i(E)$, $i=1,\ldots, N$, are disjoint, where $E$  is the attractor of $\Phi$.

It is easy to see that the SSC implies the OSC. Indeed if $E$ is the attractor of an IFS $\Phi=\{\phi_i\}_{i=1}^N$ on $\R^d$ satisfying the SSC,  letting $$V_\epsilon(E)=\{x\in \R^d:\; |x-y|<\epsilon \mbox{ for some }y\in E\}$$ denote the $\epsilon$-neighborhood of $E$, then one can check that $\phi_i( V_\epsilon(E))$ ($i=1,\ldots, N$) are disjoint subset of $V_\epsilon(E)$ for all $$0<\epsilon<\min_{i\neq j}{\rm dist}(\phi_i(E),\phi_j(E)),$$
so $\Phi$ satisfies the OSC. Meanwhile there are many examples of IFSs which satisfy the OSC but not the SSC, such as the IFS $\{x/5,\; (x+3)/5,\;(x+4)/5\}$ on $\R$. There are some equivalent conditions for the OSC \cite{BandtGraf1992, Schief1994}, but usually it is difficult to check whether these conditions hold for a given IFS.  We emphasize that a self-similar set can be generated by many different IFSs of similitudes, and under mild assumptions these IFSs have a rigid algebraic structure (see, e.g., \cite{FengWang2009, EKM2010, DengLau2013, DengLau2017, Algom2020}).

In this paper, we consider the following folklore  question about the separation conditions on the generating IFSs of self-similar sets.

\begin{ques}
\label{ques-1}
Are there  two IFSs $\Phi$ and $\Psi$ of  similitudes on $\R^d$ which generate the same self-similar set, such that $\Phi$ satisfies the OSC but not the SSC, and $\Psi$  satisfies the SSC?
\end{ques}

This question was first brought to us by Ka-Sing Lau and Jun Jason Luo  around 10 years ago. As they informed us, this question was also asked  by Mariusz Urba\'{n}ski in a private communication. We remark that Question \ref{ques-1} is closely related to an open question raised by Elekes, Keleti and M\'{a}th\'{e} \cite{EKM2010}; see Section \ref{S-4}.

In this paper, we are able to provide the following partially negative answer to Question \ref{ques-1}, stating that there is no such pair $(\Phi, \Psi)$ with $\Phi$ being homogeneous.  Recall that an IFS $\{\phi_i\}_{i=1}^N$ of similitudes is said to be \emph{homogeneous} if all the maps $\phi_i$ have the same contraction ratio.

\begin{thm}\label{t:OSneSS}
	Let $E$ be the attractor of an IFS $\Phi=\{\phi_i(x)=\rho R_ix+a_i\}_{i=1}^N$ on $\R^d$, where $0<\rho<1$,  $R_i$ are orthogonal transformations on $\R^d$ and $a_i\in \R^d$. Suppose that $\Phi$ satisfies the OSC, but not the SSC. Then $E$ can not be generated by any IFS of similitudes on $\R^d$ satisfying the SSC.
\end{thm}

The proof of the above theorem is a little bit delicate. For the convenience of the readers, we illustrate some rough ideas.  Suppose on the contrary that $E$ is also generated by another
IFS $\Psi=\{\psi_i\}_{j=1}^M$ of similitudes satisfying the SSC. We may assume that $r_j<\rho$ for all $j$, where $r_j$ is the contraction ratio of $\psi_j$.  For simplicity we assume that $E$ is not contained in any hyperplane of $\R^d$. Choose a sufficiently small  $\epsilon>0$ and let $\mathcal I$ denote the collection of all homogeneous generating IFSs (of similitudes on $\R^d$) of $E$ with contraction ratios lying in the interval $[\rho \epsilon, \epsilon)$. Applying a result of Elekes et al.~\cite{EKM2010} we see that $\mathcal I$ is finite (see Lemma \ref{l:finite}). For each $\Theta=\{\theta_i\}_{i\in \mathcal A}\in \mathcal I$, we can
construct $M$ homogeneous generating IFSs $\Gamma_1$, \ldots, $\Gamma_M$ of $E$ by $\Gamma_j=\{\psi_j^{-1}\circ \theta_i:\; i\in \mathcal A_j\},$ where $\mathcal A_j:=\{i\in \mathcal
A:\; \theta_i(E)\subset \psi_j(E)\}$. For each $j$ we can show that there exists a positive integer $k_j$ such that $\Phi^{k_j}\circ \Gamma_j\in \mathcal I$, where $\Phi^{k_j}$ denotes the $k_j$-th iteration of $\Phi$.  To derive a contradiction, we assign an infinite dimensional probability vector $\gamma(\Phi')$ (with finitely many non-zero entries) to each generating IFS $\Phi'$ of $E$ satisfying the OSC;  see Section \ref{S-2.1}. The term $\gamma(\Phi')$, which is called the characteristic vector of $\Phi'$,  gives a quantitative description of the intersections between the images of $E$ under the mappings in $\Phi'$. Moreover we can introduce a total order relation $\preceq$ on the collection of all such vectors.  A key observation is that
\begin{equation}
\label{e-key}\gamma(\Phi')\prec \gamma(\Phi\circ \Phi')
\end{equation}
for each generating IFS $\Phi'$ of $E$ satisfying the OSC (see Lemma \ref{l:inc}). Now return back to the aforementioned IFSs $\Theta$ and $\Gamma_j$ ($j=1,\ldots, M$). From the definition of characteristic vector, we obtain an identity $\gamma(\Theta)=\sum_{j=1}^M r_j^s\gamma(\Gamma_j)$, where $s$ is the Hausdorff dimension of $E$ which satisfies  $\sum_{j=1}^M r_j^s=1$. As $\mathcal I$ is finite, we may choose $\Theta\in \mathcal I$ so that $\gamma(\Theta)$ is the largest in the sense that $\gamma(\Theta')\preceq \gamma(\Theta)$ for all $\Theta'\in \mathcal I$. Then by \eqref{e-key},
$$\gamma(\Theta)=\sum_{j=1}^M r_j^s\gamma(\Gamma_j)\prec \sum_{j=1}^M r_j^s\gamma(\Phi^{k_j}\circ \Gamma_j)\preceq \sum_{j=1}^M r_j^s\gamma(\Theta)=\gamma(\Theta),$$
leading to a contradiction $\gamma(\Theta)\prec \gamma(\Theta)$.

The paper is organized as follows. In Section \ref{sc:pre} we introduce the definition of characteristic vector for each  IFS of similitudes, and prove Lemma \ref{l:inc}. In Section \ref{sc:proof} we prove Theorem \ref{t:OSneSS}. In Section \ref{S-4} we give a final remark.

\section{Characteristic vectors of IFSs and a key property}\label{sc:pre}

\subsection{Characteristic vectors of IFSs of similitudes}\label{S-2.1}
In this subsection, we will assign an infinite-dimensional probability vector to each IFS of similitudes on $\R^d$.

To begin with, let $\Phi=\{\phi_i(x)=\rho_iR_ix+a_i\}_{i=1}^N$ be an IFS of similitudes on $\R^d$  and let $E=E_\Phi$ be the attractor of $\Phi$.  We introduce a binary relation $\sim_\Phi$ on the set $V:=V_\Phi=\{1,\ldots, N\}$ by
$$\mbox{$i\sim_\Phi j$ \; if  $\phi_i(E)\cap \phi_j(E)\neq \emptyset$}.
$$
 We say that $i$ is {\it adjacent} to $j$ if $i\sim_\Phi j$. Moreover, we say that $i$ is {\it connected} to $j$, if there exist $x_1,\ldots,x_n\in \{1,\ldots, N\}$ with $x_1=i$ and $x_n=j$ such that $x_k$ is adjacent to $x_{k+1}$ for  $k=1,\ldots, n-1$. A subset $V'$ of $V$ is said to be {\it connected} if each two elements of $V'$ are connected to each other.

Clearly,  $V$ can be written as the union of a number of disjoint maximal connected subsets, each of them is called a {\it connected component} of $V$.  Moreover a connected component of $V$ with cardinality $n$ is called an {\it $n$-component} of $V$.

It is direct to check that  $\Lambda\subset V$ is an $n$-component of $V$ if and only if the following two properties hold:

\begin{itemize}
\item[(i)] $C_\Lambda\cap C_{V\setminus\Lambda}=\emptyset$, where $C_\Lambda:=\bigcup_{i\in\Lambda}\phi_i(E)$.
\item[(ii)] $C_{\Lambda'}\cap C_{V\setminus\Lambda'}\ne\emptyset$
for every nonempty proper subset $\Lambda'$ of $\Lambda$.
\end{itemize}

Let $s$ denote the similarity dimension of $\Phi$, i.e., $s$ is the unique number so that  $\sum_{i=1}^N\rho_i^s=1$. We define an infinite dimensional vector $\gamma(\Phi)=(\gamma^\Phi_n)_{n=1}^\infty$ by
\[ \gamma^\Phi_n=\sum_{\Lambda\colon \text{\rm $\Lambda$ is a $n$-component of V}}\sum_{i\in \Lambda}\rho^s_i,\]
where we adopt the convention that $\gamma^\Phi_n=0$ if $V$ has no any $n$-component. Clearly $\gamma^\Phi_n=0$ for all $n>N$ and $\sum_{n=1}^\infty\gamma^\Phi_n=1$.

\begin{de}
We call $\gamma(\Phi)=(\gamma^\Phi_n)_{n=1}^\infty$ the characteristic vector of $\Phi$.
\end{de}

Below we give some simple examples to illustrate the characteristic vector $\gamma(\Phi)$ of $\Phi$.

\begin{ex}
	If $\Phi$ satisfies the SSC, then $\gamma(\Phi)=(1,0,\ldots)$.
\end{ex}
\begin{ex}
	Suppose that the attractor $E=E_\Phi$ of $\Phi=\{\phi_i\}_{i=1}^N$  is connected, then $\gamma^\Phi_N=1$ and $\gamma^\Phi_n=0$ for all $n\neq N$.
\end{ex}
\begin{ex}\label{e:145}
	Let  $\Phi=\{\phi_i\}_{i=1}^3$ be an IFS on $\R$ defined by $$\phi_1(x)=\frac{x}{5},\quad \phi_2(x)=\frac{x+3}{5},\quad \phi_3=\frac{x+4}{5}.$$ Then
	$ \gamma(\Phi)=(1/3,2/3,0,\ldots)$.
\end{ex}

\subsection{A key property of characteristic vectors}
In order to compare two characteristic vectors of IFSs, we write
$$\Omega=\left\{(x_n)_{n=1}^\infty\in \R^{\N}:\; \mbox{$x_n\neq 0$ for at most finitely many $n$}\right\}.
$$
Clearly, $\Omega$ is a vector space over $\R$ and $\gamma(\Phi)\in \Omega$ for each IFS $\Phi$ of similitudes on $\R^d$.
Now we introduce a relation $\prec$ on $\Omega$ by
$$
x\prec y\;\;\mbox{if there exists }m\geq 1\mbox{ such that }x_m<y_m \mbox{ and $x_n=y_n$ for all $n> m$},
$$
where $x=(x_n)_{n=1}^\infty$ and $y=(y_n)_{n=1}^\infty$.
We write $x\preceq y$ if $x\prec y$ or $x=y$. Clearly $\Omega$ is totally ordered in the sense that for any $x,y\in\Omega$, we have either $x=y$, or $x\prec y$, or $y\prec x$. Moreover, if $x\preceq y$ then $ax\preceq ay$ for all $a>0$; and if $x\preceq y$, $u\preceq v$ then $x+u\preceq y+v$.

For two IFSs~$\Phi=\{\phi_i\}_{i=1}^N$ and~$\Psi=\{\psi_j\}_{j=1}^M$ on $\R^d$, the composition of~$\Phi$ and~$\Psi$ is a new IFS on $\R^d$ given by
\[ \Phi\circ\Psi=\{\phi_i\circ\psi_j\colon 1\leq i\leq N,\;1\leq j\in M\}. \]
We begin with a simple lemma.
\begin{lem}\label{l:inc1}
	Let $\Phi=\{\phi_i\}_{i=1}^N$ and~$\Psi=\{\psi_j\}_{j=1}^M$ be two IFSs of similitudes on $\R^d$ satisfying the OSC. Suppose that $\Phi$ and $\Psi$ generate the same self-similar set~$E$, i.e., $E_\Phi=E_\Psi=E$. Then $\Phi\circ\Psi$ is also a generating IFS of $E$ satisfying the OSC.
\end{lem}
\begin{proof}
Since $E_\Phi=E_\Psi=E$,  we have $E=\bigcup_{i=1}^N\phi_i(E)=\bigcup_{j=1}^M\psi_i(E)$. It follows that
$$
\bigcup_{i=1}^N\bigcup_{j=1}^M\phi_i(\psi_j(E))=\bigcup_{i=1}^N\phi_i\left(\bigcup_{j=1}^M\psi_j(E)\right)
=\bigcup_{i=1}^N\phi_i(E)=E,
$$
which implies $E_{\Phi\circ \Psi}=E$.

To see that $\Phi\circ \Psi$ satisfies the OSC, let $s$ be the Hausdorff dimension of $E$. Clearly,  $\Phi$ and $\Psi$ have the same similarity dimension $s$ and $\mathcal H^s(E)>0$, where $\mathcal H^s$ stands for the $s$-dimensional Hausdorff measure; see \cite{Falconer2003, Hut1981}. It follows that  $\Phi\circ \Psi$ also has the similarity dimension $s$. Since $\mathcal H^s(E)>0$, by \cite[Theorem 2.1]{Schief1994} $\Phi\circ \Psi$ satisfies the OSC.
\end{proof}

The following lemma plays a key role in the proof of Theorem \ref{t:OSneSS}.
\begin{lem}\label{l:inc}
	Let $\Phi=\{\phi_i\}_{i=1}^N$ and~$\Psi=\{\psi_j\}_{j=1}^M$ be two IFSs of similitudes on $\R^d$ satisfying the OSC. Suppose that $\Phi$ and $\Psi$ generate the same self-similar set~$E$.   Moreover, suppose that $\Phi$ does not satisfy the SSC.  Then $\gamma(\Psi)\prec\gamma(\Phi\circ\Psi)$.
\end{lem}

\begin{proof} By Lemma \ref{l:inc1}, $\Phi\circ \Psi$ is a generating IFS of $E$ satisfying the OSC.
Write $$V_\Psi=\{1,\ldots, M\} \mbox { and }V_{\Phi\circ \Psi}=\{(i,j):\; i=1,\ldots, N,\; j=1,\ldots, M\}.$$

Define the binary relations $\sim_\Psi$ and $\sim_{\Phi\circ \Psi}$  on $V_\Psi$ and  $V_{\Phi\circ \Psi}$ respectively as in Section \ref{S-2.1}.
For $n\geq 1$, let  $\mathcal C_n$ denote  the collection of $n$-components of $V_\Psi$ with respect to $\sim_\Psi$, and $\mathcal D_n$ the collection of $n$-components of $V_{\Phi\circ \Psi}$ with respect to $\sim_{\Phi\circ \Psi}$; see Section \ref{S-2.1} for the relevant definitions.
	
Let us first give a simple observation. 	Suppose that two elements $j,k\in V_\Psi$ are connected to each other, i.e., there exist $j_1,\ldots, j_s\in V_\Psi$, with $j_1=j$ and $j_s=k$,  such that $\psi_{j_\ell}(E)\cap \psi_{j_{\ell+1}}(E)\neq \emptyset$ for all $\ell=1,\ldots, s-1$. Clearly for each $i\in \{1,\ldots, N\}$, $\phi_i\circ \psi_{j_\ell}(E)\cap \phi_i\circ \psi_{j_{\ell+1}}(E)\neq \emptyset$ for all $\ell=1,\ldots, s-1$; therefore the elements $(i,j)$ and $(i, k)$ in $V_{\Phi\circ \Psi}$ are connected to each other.

According to the above observation, for every $n$-component~$\Lambda$ of~$V_\Psi$ and $i\in \{1,\ldots, N\}$, the set $\{i\}\times  \Lambda$ is connected with respect to $\sim_{\Phi\circ \Psi}$, hence it is either an $n$-component of~$V_{\Phi\circ\Psi}$ or a proper subset of  an $n_1$-component of~$V_{\Phi\circ\Psi}$ for some  $n_1>n$.

We claim that for some $n\in \N$,  there are an $n$-component $\Lambda$ of~$V_\Psi$ and $i\in \{1,\ldots, N\}$ such that the set $\{i\}\times  \Lambda$ is not an $n$-component of~$V_{\Phi\circ\Psi}$. To prove this claim we use contradiction. Suppose on the contrary that the claim is false. Then for each $i\in \{1,\ldots, N\}$, $n\in \N$ and $\Lambda\in \mathcal C_n$,  we have $\{i\}\times  \Lambda \in \mathcal D_n$, which implies that  for each $1\leq i'\leq N$ with $i'\neq i$,
	\[ \phi_i\left(\bigcup_{j\in \Lambda}\psi_j(E)\right)\cap\phi_{i'}(E)=\phi_i\left(\bigcup_{j\in \Lambda}\psi_j(E)\right)\cap \phi_{i'}\left(\bigcup_{k=1}^M\psi_k(E)\right)=\emptyset. \]
Taking the union over all connected components of $V_\Psi$ yields that for $i,i'\in \{1,\ldots, N\}$ with $i\neq i'$,
	\[ \phi_i(E)\cap\phi_{i'}(E)=\phi_i\left(\bigcup_{\Lambda\in \bigcup_{n\ge1}\mathcal C_n}\bigcup_{j\in \Lambda}\psi_j(E)\right)\cap \phi_{i'}(E)=\emptyset, \]
	which contradicts the assumption that  $\Phi=\{\phi_i\}_{i=1}^N$ does not satisfy the SSC. This completes the proof of the claim.
	
	Let $n_0$ be the largest integer such that there exist an $n_0$-component  $\Lambda_0$ of $V_\Psi$ and $i_0\in \{1,\ldots N\}$ such that the set
	$\{i_0\}\times  \Lambda_0$ is not an $n_0$-component of $V_{\Phi\circ \Psi}$.  Then as we pointed out above, $\{i_0\}\times  \Lambda_0$ is a proper subset of an $n_1$-component $D_0$ of $V_{\Phi\circ \Psi}$ for some $n_1>n_0$.  Clearly $D_0$ is not of the form $\{i\}\times \Lambda$ with $1\leq i\leq N$ and $\Lambda\in \mathcal C_{n_1}$.
	
	Now for $i\in \{1,\ldots, N\}$ and $j\in \{1,\ldots, M\}$, let $\rho_i$ and  $r_j$  denote the contraction ratios of $\phi_i$ and $\psi_j$, respectively.  Let $s$ denote the Hausdorff dimension of $E$.  Then for each  $n\geq n_1$ and $\Lambda\in \mathcal C_n$, $\{i\}\times \Lambda\in \mathcal D_n$ for each $1\leq i\leq N$. It follows that  for $n\geq n_1$,
\begin{equation}\label{eq:k+1>=k}
	\gamma^{\Phi\circ\Psi}_n=\sum_{D\in\mathcal D_n}\sum_{(i, j)\in D}(\rho_ir_j)^s
	\ge\sum_{i=1}^N \sum_{\Lambda\in\mathcal C_n}\sum_{j\in \Lambda} (\rho_i r_j)^s= \sum_{\Lambda\in\mathcal C_n}\sum_{j\in \Lambda}  r_j^s=\gamma^{\Psi}_n.
\end{equation}
	Recall that $\mathcal D_{n_1}$ contains an element $D_0$ which is not of the form $\{i\}\times \Lambda$ with $1\leq i\leq N$ and $\Lambda\in \mathcal C_{n_1}$.  Hence
\begin{equation}\label{eq:k+1>k}
\begin{split}
	\gamma^{\Phi\circ\Psi}_{n_1}&=\sum_{D\in\mathcal D_{n_1}}\sum_{(i, j)\in D}(\rho_ir_j)^s\\
	&\geq \sum_{(i',j')\in D_0}(\rho_{i'}r_{j'})^s+ \sum_{i=1}^N \sum_{\Lambda\in\mathcal C_{n_1}}\sum_{j\in \Lambda} (\rho_i r_j)^s\\
		&= \sum_{(i',j')\in D_0}(\rho_{i'}r_{j'})^s+\gamma^{\Psi}_{n_1}\\
	&>\gamma^{\Psi}_{n_1}.
\end{split}
\end{equation}
	Combining~\eqref{eq:k+1>=k} and~\eqref{eq:k+1>k} yields  that $\gamma(\Psi)\prec\gamma(\Phi\circ\Psi)$.
\end{proof}

%\begin{ex}
%	Let $\Phi$ be as in Example~\ref{e:145}. A direct calculation shows that
%	\[ \gamma(\Phi)=(1/3,2/3,0,\ldots) \prec \gamma(\Phi\circ\Phi)=(2/9,4/9,1/3,0,\ldots). \]
%\end{ex}

\section{The proof of Theorem~\ref{t:OSneSS}}\label{sc:proof}

In this section we prove Theorem~\ref{t:OSneSS}.
For a compact set $E\subset \R^d$,  let $\mathcal I_E$ denote the collection of all homogeneous IFSs of similitudes on $\R^d$ that generate $E$ and satisfy the OSC. Write
\begin{equation}
	\mathcal I_{E, a,b}=\{ \Phi\in \mathcal I_E:\; a\leq \rho_\Phi<b\}
\end{equation}
for $0<a<b<1$, where $\rho_\Phi$ denotes the common contraction ratio of the maps in $\Phi$.
\begin{lem}\label{l:finite}
	Let $E$ be the attractor of an IFS of similitudes on $\R^d$ satisfying the SSC. Suppose that $E$ is not contained in any hyperplane of~$\R^d$.  Then for any $0<a<b<1$, $\mathcal I_{E,a,b}$ is a finite set.
\end{lem}
\begin{proof}  Let ${\rm Sim}(d)$ denote the collection of all similarity maps of $\R^d$. Notice that each element $\phi$ in ${\rm Sim}(d)$ is an affine map on $\R^d$ which is of the form $\phi(x)=Ax+a$, where $A$ is a $d\times d$ matrix and $a\in \R^d$,  hence $\phi$ can be viewed as a point in $\R^{d^2+d}$.
Therefore ${\rm Sim}(d)$ can be viewed as a metric subspace of $\R^{d^2+d}$, where we endow $\R^{d^2+d}$ with the usual Euclidean metric.

For $\delta>0$, write
	\[ \mathcal S_{E, \delta}=\left\{ \phi\in  {\rm Sim}(d):\; \phi(E)\supset E \mbox{ and } \rho_\phi\leq \delta\right\}, \]
	where $\rho_\phi$ denotes the similarity ratio of $\phi$. 	To prove the conclusion of the lemma, it suffices to show that $\mathcal S_{E, \delta}$ is finite for every $\delta>0$.
	
Since  $E$ is a compact subset of $\R^d$, it follows that  $\mathcal S_{E, \delta}$ is a compact subset of  ${\rm Sim}(d)$ for each $\delta>0$.
	 Meanwhile by \cite[Proposition~4.3(i)]{EKM2010}, under the assumptions of the lemma on $E$,    $\{ \phi\in  {\rm Sim}(d):\; \phi(E)\supset E\}$ is a discrete subset of  ${\rm Sim}(d)$. Hence $\mathcal S_{E,\delta}$ is both discrete and compact, so it is finite.
\end{proof}

Now we are ready to prove Theorem \ref{t:OSneSS}.

\begin{proof}[Proof of Theorem~\ref{t:OSneSS}] We first assume that $E$ is not contained in a hyperplane of $\R^d$. To prove the theorem we use contradiction. Suppose on the contrary that $E$ can be generated by an IFS  $\Psi=\{\psi_j\}_{j=1}^M$ of similitudes on $\R^d$ which satisfies the SSC. Below we derive a contradiction.

Replacing $\Psi$ by its $n$-th iteration $\Psi^n:=\underbrace{\Psi\circ\cdots \circ \Psi}_n$ for a large $n$ if necessary, we may assume that $r_j<\rho$ for all $1\leq j\leq M$, where $r_j$ denotes the contraction ratio of $\psi_j$. Set
\begin{equation}\label{eq:dist}
	\delta:=\min_{j\ne j'}\dist\left( \psi_j(E),\psi_{j'}(E) \right).
\end{equation}
Since $\Psi$ satisfies the SSC, we have $\delta>0$.  Clearly $\delta<{\rm diam}(E)$.

Pick a large positive integer $\ell$ so that $\rho^\ell<\min_{1\leq j\leq M}r_j$. Then fix a small $\epsilon>0$ such that
\begin{equation}
\label{e-epsilon}
\epsilon\leq \frac{\rho^\ell\delta}{{\rm diam}(E)}.
\end{equation}
Clearly $\epsilon<\rho^\ell$ since $\delta<{\rm diam}(E)$. Let $\mathcal I:= \mathcal I_{E, \rho \epsilon, \epsilon}$ denote the collection of homogeneous generating IFSs $\Phi'$ on $\R^d$ of $E$ satisfying the OSC and $\rho \epsilon\leq \rho_{\Phi'}< \epsilon$. By Lemma~\ref{l:finite},  $\mathcal I$ is a finite set. Let $k$ be the unique integer so that $\rho\epsilon\leq \rho^k< \epsilon$. Then $k>\ell\geq 1$ and $\Phi^k\in \mathcal I$, so $\mathcal I$ is non-empty.  Since $\mathcal I$ is a nonempty finite set, there exists  $\Theta\in \mathcal I$ whose characteristic vector $\gamma(\Theta)$ is the largest in the sense that $\gamma(\Phi')\preceq \gamma(\Theta)$ for all $\Phi'\in \mathcal I$.

  Keep in mind that $\rho_{\Theta}\in[\rho\epsilon,\epsilon)$ since $\Theta\in \mathcal I$. Write $\Theta=\{\theta_i\}_{i\in \mathcal A}$. By \eqref{e-epsilon},
  \begin{equation}
  \label{e-t1}
  {\rm diam}(\theta_i(E))=\rho_\Theta{\rm diam}(E)<  \epsilon  {\rm diam}(E)\leq \rho^{\ell} \delta<\delta,
  \end{equation}
 and
 \begin{equation}
 \label{e-t2}
 \rho_{\Theta}<\epsilon<\rho^\ell<\min_{1\leq j\leq M}r_j.
 \end{equation}
  Since $\bigcup_{i\in \mathcal A}\theta_i(E)=E=\bigcup_{j=1}^M\psi_j(E)$, by \eqref{eq:dist} and  \eqref{e-t1} we see that if $\theta_i(E)\cap \psi_j(E)\neq \emptyset$ then $\theta_i(E)\subset \psi_j(E)$. Due to this, we can partition  $\mathcal A$  into $M$ disjoint subsets $\mathcal A_1$,\ldots, $\mathcal A_M$ by setting
  $$\mathcal A_j:= \{i\in \mathcal A\colon \theta_i(E)\subset \psi_j(E)\},\quad j=1,\ldots, M,$$
  and moreover,
  \begin{equation}
  \label{e-t3}
  \bigcup_{i\in \mathcal A_j}\theta_i(E)=\psi_j(E), \quad j=1,\ldots, M.
  \end{equation}
For $j=1,\ldots, M$, let
\[ \Gamma_{j}=\left\{ \psi_j^{-1}\circ\theta_i:\; i\in \mathcal A_j\right\}. \]
By \eqref{e-t3} and \eqref{e-t2}, for each $1\leq j\leq M$, $\Gamma_j$ is a homogeneous generating IFS of $E$ satisfying the OSC, and its contraction ratio $\rho_{\Gamma_j}$ satisfies
\begin{equation}\label{eq:ratio}
	\rho_{\Gamma_j}=r_j^{-1}\rho_{\Theta}\ge r_j^{-1}\rho\epsilon> \epsilon,
\end{equation}
where we have used $\rho_\Theta\geq \rho\epsilon$ and $r_j<\rho$ in the last two inequalities. For each $1\leq j\leq M$, let $k_j$ be the unique integer such that
$$\rho^{k_j} 	\rho_{\Gamma_j}\in [\rho\epsilon,\epsilon),$$
 then $k_j\geq 1$ by \eqref{eq:ratio}, and  $\Phi^{k_j}\circ \Gamma_j\in \mathcal I$, where $\Phi^{k_j}$ denotes the $k_j$-th iteration of $\Phi$.

Next we compare the characteristic vectors of the IFSs $\Theta$ and $\Gamma_j$ ($j=1,\ldots, M$).
Since $\Psi$ satisfies the SSC, by  \eqref{e-t3}  we see that $\theta_i(E)\cap \theta_{i'}(E)=\emptyset$ if $i\in \mathcal A_j$ and $i'\in \mathcal A_{j'}$ for some $j\neq j'$.  It follows that for  $n\in \N$, every $n$-component of $\mathcal A$ (with respect to $\sim_\Theta$) is totally contained in $\mathcal A_j$ for some $j$. Hence for each $n\in \N$,
\begin{equation}
\label{e-t5}
\begin{split}
\gamma^\Theta_n&=\sum_{\rm \Lambda:\;  \text{$\Lambda$ is an $n$-component of $\mathcal A$}}\; \sum_{i\in \Lambda}\rho_{\Theta}^s\\
&=\sum_{j=1}^M\;\sum_{\rm \Lambda:\;  \text{$\Lambda$ is an $n$-component of $\mathcal A_j$}}\; \sum_{i\in \Lambda}\rho_{\Theta}^s.
\end{split}
\end{equation}
Meanwhile by \eqref{e-t3}, it is easy to see that each $n$-component of $\mathcal A_j$ with respect to $\sim_\Theta$ is an $n$-component of  $\mathcal A_j$ with respect to $\sim_{\Gamma_j}$, and vice versa.  It follows that for $j=1,\ldots, M$,
\begin{equation}
\label{e-t6}
\gamma^{\Gamma_j}_n=\sum_{\rm \Lambda:\;  \text{$\Lambda$ is an $n$-component of $\mathcal A_j$}} \; \sum_{i\in \Lambda}(r_j^{-1}\rho_{\Theta})^s.
\end{equation}
Combining \eqref{e-t5} with \eqref{e-t6} yields that $\gamma^\Theta_n=\sum_{j=1}^M r_j^s \gamma^{\Gamma_j}_n$ for each $n\in \N$. Hence
\begin{equation}\label{eq:ave}
	\gamma(\Theta)=\sum_{j=1}^M r_j^s \gamma(\Gamma_j).
\end{equation}

Recall that for each $1\leq j\leq M$, there exists $k_j\in \N$ such that $\Phi^{k_j}\circ \Gamma_j\in \mathcal I$. This implies that $\gamma(\Phi^{k_j}\circ\Gamma_j)\preceq \gamma(\Theta)$. Since $\Phi$ satisfies the OSC but not the SSC, by Lemma \ref{l:inc} we have
$$
 \gamma(\Gamma_j)\prec \gamma(\Phi\circ \Gamma_j)\prec\cdots\prec \gamma(\Phi^{k_j}\circ \Gamma_j)\preceq \gamma(\Theta)
$$
for $j=1,\ldots, M$.  Combining this with \eqref{eq:ave} yields that $$\gamma(\Theta)\prec \sum_{j=1}^M r_j^s \gamma(\Theta)=\gamma(\Theta),$$ which leads to a contradiction. This proves the theorem under the assumption that $E$ is not contained in a hyperplane of $\R^d$.

Finally we consider the general case when $E$ may be contained in a hyperplane of $\R^d$.  Suppose on the contrary that $E$ can be generated by an IFS  $\Psi=\{\psi_j\}_{j=1}^M$ of similitudes on $\R^d$ which satisfies the SSC.  Let $H$ be the affine hull of $E$, i.e., $H$ is the affine subspace of $\R^d$ with the smallest dimension that covers $E$.  Then it is direct to check that the following properties hold:
\begin{itemize}
\item[(i)] $E$ is not contained in a proper affine subspace of $H$;
\item[(ii)]  $\phi_i(H)=H$, $\psi_j(H)=H$ for all $1\leq i\leq N$ and $1\leq j\leq M$;
\item[(iii)]  Let $\widehat{\phi}_i$, $\widehat{\psi}_j$ denote the restrictions of $\phi_i$ and $\psi_j$ on $H$. Then $\widehat{\Phi}=\{\widehat{\phi}_i\}_{i=1}^N$ and $\widehat{\Psi}=\{\widehat{\psi}_j\}_{j=1}^M$ are IFSs of similitudes on $H$ that generate $E$, and moreover, $\widehat{\Phi}$ is a homogeneous IFS satisfying the OSC, and $\widehat{\Psi}$ satisfies the SSC.
\end{itemize}
Then we can derive a contradiction by following the previous argument (in which $\Phi$, $\Psi$ and $\R^d$  are replaced by $\widehat{\Phi}$, $\widehat{\Psi}$ and $H$, respectively). This completes the proof of the theorem.
\end{proof}

\section{A final remark}\label{S-4}

We remark that Question \ref{ques-1} is closely related to the following question raised by Elekes, Keleti and M\'{a}th\'{e}.

\begin{ques}[{\cite[Question~9.3]{EKM2010}}]
\label{ques-2}
Let $E\subset \R^d$  be a self-similar set generated by an IFS $\Psi=\{\psi_j\}_{j=1}^M$ satisfying the SSC and let $f$ be a similitude
such that $f(E)\subset E$. Does this imply that $f(E)$ is a relative open set in $E$ (or in other
words $f(E)$ is a finite union of elementary pieces of $E$)?
\end{ques}

Here an elementary piece of $E$ means a set of the form $\psi_{j_1}\circ \cdots\circ \psi_{j_m}(E)$, where $j_1,\ldots, j_m\in \{1,\ldots, M\}$.
To our best knowledge, so far Question \ref{ques-2} still remains open. We remark that  an affirmative answer to Question \ref{ques-2} would yield a negative answer to Question \ref{ques-1}.   To see this, suppose that $E$ is the attractor of an IFS $\Psi=\{\psi_j\}_{j=1}^M$ of similitudes satisfying the SSC, and that the answer to Question \ref{ques-2} is affirmative.  Notice that the collection of elementary pieces of $E$ (with respect to the IFS $\Psi$) has the following net structure: for any two given elementary pieces $E_1$ and $E_2$, one has either $E_1\cap E_2=\emptyset$, or $E_1\subset E_2$, or $E_1\supset E_2$. Hence if $f(E)\subset E$ for a similitude $f$, then $f(E)$ is a finite union of disjoint elementary pieces of $E$. As a consequence, if $f_1$ and $f_2$ are two similitudes mapping $E$ into itself,  then either $f_1(E)\cap f_2(E)=\emptyset$, or $f_1(E)\cap f_2(E)$ contains an elementary piece of $E$ which implies  that $\mathcal H^s(f_1(E)\cap f_2(E))>0$, where $s$ denotes the Hausdorff dimension of $E$. Now if $\Phi=\{\phi_i\}_{i=1}^N$ is another generating IFS (of similitudes) of $E$ satisfying the OSC, then $\mathcal H^s(\phi_i(E)\cap \phi_j(E))=0$ for all $i\neq j$ (see \cite{Hut1981}) which forces that  $\phi_i(E)\cap \phi_j(E)=\emptyset$ for all $i\neq j$, that is, $\Phi$ satisfies the SSC. Therefore there is no generating IFS (of similitudes) of $E$ which satisfies the OSC but not the SSC.

\bigskip

\noindent {\bf Acknowledgements}.  Feng was partially supported by the General Research Funds (CUHK14301017, CUHK14303021)  from the Hong Kong Research Grant Council.  Ruan was partially  by  NSFC grant 11771391,  ZJNSF grant LY22A010023 and the Fundamental Research Funds for the Central Universities of China (grant 2021FZZX001-01).    Xiong was partially supported by NSFC grant 11871227, and Guangdong Basic and Applied Basic Research Foundation (project 2021A1515010056).

\end{document}